\documentclass{scrartcl}

\usepackage[naustrian, english]{babel}
\usepackage{hyperref}
\usepackage[T1]{fontenc}
\usepackage{lmodern}
\usepackage[utf8]{inputenc}
\usepackage{amsmath}
\usepackage{amssymb}
\usepackage{amsthm}
\usepackage[babel]{csquotes}
\usepackage{tikz}
\usepackage{enumitem}
\usepackage{bm}
\usepackage{upgreek}
\usepackage{subfigure}

\newcommand{\p}[1]{\ensuremath{\mathord{\left(#1\right)}}}
\newcommand{\abs}[1]{\ensuremath{\left\vert#1\right\vert}}
\newcommand{\set}[1]{\ensuremath{\left\lbrace #1 \right\rbrace}}
\newcommand{\setcond}[2]{\ensuremath{\left\lbrace #1 \,\middle\vert \, #2 \right\rbrace}}
\newcommand{\inpr}[2]{\ensuremath{\left\langle #1, #2 \right\rangle}}
\newcommand{\RR}{\ensuremath{\mathbb{R}}}

\newcommand{\clint}[2]{\ensuremath{\left[#1, #2\right]}}
\newcommand{\opint}[2]{\ensuremath{\left(#1, #2\right)}}
\newcommand{\coint}[2]{\ensuremath{\left[#1, #2\right)}}

\newcommand{\defeq}{\ensuremath{\mathrel{\mathop:}=}}

\newcommand{\RRc}{\ensuremath{\overline{\mathbb{R}}}}
\newcommand{\toset}{\ensuremath{\rightrightarrows}}
\newcommand{\Hh}{\mathcal{H}}

\newcommand{\ee}{\mathrm{e}}

\newcommand{\norm}[1]{\left\Vert #1 \right\Vert}
\newcommand{\weakto}{\rightharpoonup}

\newcommand{\dd}[1]{\mathop{\mathrm{d} #1}}

\newcommand{\Id}{\ensuremath{\mathrm{Id}}}

\let \cb \c

\DeclareMathOperator{\dom}{dom}
\DeclareMathOperator{\zer}{Zer}

\DeclareMathOperator{\cl}{cl}

\DeclareMathOperator*{\argmin}{arg\,min}

\DeclareMathOperator{\Graph}{Graph}
\DeclareMathOperator{\Dom}{Dom}
\DeclareMathOperator{\Ran}{Ran}
\DeclareMathOperator{\Proj}{Proj}
\DeclareMathOperator{\Prox}{Prox}

\theoremstyle{plain}
\ifcsdef{lemma}{\relax}{\newtheorem{lemma}{Lemma}}
\ifcsdef{theorem}{\relax}{\newtheorem{theorem}{Theorem}}
\ifcsdef{corollary}{\relax}{\newtheorem{corollary}{Corollary}}
\ifcsdef{proposition}{\relax}{}

\theoremstyle{definition}
\ifcsdef{definition}{\relax}{\newtheorem{definition}{Definition}}
\ifcsdef{example}{\relax}{}
\ifcsdef{remark}{\relax}{\newtheorem{remark}{Remark}}
\ifcsdef{problem}{\relax}{}
\ifcsdef{algorithm}{\relax}{}

\title{A forward-backward-forward differential equation and its asymptotic properties}
\author{Sebastian Banert\thanks{University of Vienna, Faculty of Mathematics, Oskar-Morgenstern-Platz 1, 1090 Vienna, Austria, sebastian.banert@univie.ac.at} \and 
Radu Ioan Bo\c{t}\thanks{University of Vienna, Faculty of Mathematics, Oskar-Morgenstern-Platz 1, 1090 Vienna, Austria, radu.bot@univie.ac.at.}}

\begin{document}
\maketitle

\noindent \textbf{Abstract.}
In this paper, we approach the problem of finding  the zeros of the sum of a maximally monotone operator and a monotone and Lipschitz continuous one in a real Hilbert space via an implicit forward-backward-forward dynamical system
with nonconstant relaxation parameters and stepsizes of the resolvents. Besides proving existence and uniqueness of strong global solutions for the differential equation under consideration, 
we show weak convergence of the generated trajectories and, under strong monotonicity assumptions, strong convergence with exponential rate. In the particular setting of minimizing the sum of a proper, convex and
lower semicontinuous function with a smooth convex one, we provide a rate for the convergence of the objective function along the ergodic trajectory to its minimum value.\vspace{1ex}

 \vspace{1ex}


\section{Introduction}\label{sec1}
In this paper, we address the monotone inclusion problem
\begin{equation}\label{monpr}
\mbox{find} \ \bar x \in \Hh \ \mbox{such that} \ 0 \in A\bar x + B\bar x, 
 \end{equation}
where $\Hh$ is a real Hilbert space, $A: \Hh \toset \Hh$ is a maximally monotone operator and $B: \Hh \rightarrow \mathbb{R}$ is a monotone and $\frac{1}{\beta}$-Lipschitz continuous operator for $\beta > 0$, by means of 
the dynamical system of equations
\begin{equation}\label{eq:CTseng} \left\{
\begin{array}{rl}
  z\p{t} \!\!\!&= J_{\gamma\p{t} A} \p{x\p{t} - \gamma\p{t} Bx\p{t}} \\
  0   \!\!\!&= \dot x\p{t} + x\p{t} - z\p{t} - \gamma\p{t} Bx\p{t} + \gamma\p{t} Bz\p{t}\\
  x\p{0} \!\!\!&= x_0, 
\end{array}\right.  
\end{equation}
where $\gamma: \coint{0}{+\infty} \to \opint{0}{\beta}$ is a Lebesgue measurable function, $x_0 \in \Hh$ and $J_{\gamma(t)A}$ denotes the resolvent 
of the operator $\gamma(t)A$ for every $t \in [0,+\infty)$.

The pioneering work \cite{CrandallPazy:1969} of Crandall and Pazy represented a cornerstone in the study of dynamical systems governed
by maximally monotone operators in Hilbert spaces, as it addressed questions like the existence and uniqueness of
solution trajectories and it related the latter to the theory of semi-groups of nonlinear contractions. Brezis has studied in \cite{Brezis:1971} the asymptotic behavior of the 
trajectories whenever the underlying operator is the convex subdifferential and Bruck proved in \cite{Bruck:1975} that a similar asymptotic convergence analysis can be made also in the general
case involving an arbitrary maximally monotone operator. 

Dynamical systems governed by maximally monotone operators  are recognized as valuable tools for studying numerical algorithms for monotone inclusions and optimization problems 
obtained by time discretization of the continuous dynamics (cf. \cite{PeypouquetSorin:2010}). In this context we want to refer to the discrete forward-backward-forward algorithm
(see \cite{BauschkeCombettes:2011, Tseng:2000}) which generates for an initial point $x_0 \in \Hh$ and a sequence of stepsizes $(\gamma_n)_{n\geq 0} \subseteq (0,\beta)$, 
via the iterative scheme
\begin{equation}\label{eq:DTseng} 
(\forall n \geq 0) \ \left\{
\begin{aligned}
  z_n &:= J_{\gamma_n A} \p{x_n - \gamma_n Bx_n}  \\
  x_{n+1} &:= z_n + \gamma_n \p{Bx_n - Bz_n}, 
\end{aligned}
\right.
\end{equation}
two sequences $(x_n)_{n\geq 0}$ and $(z_n)_{n\geq 0}$ that converge to a solution of the monotone inclusion problem \eqref{monpr}.

Since they provide a deep understanding of the related discrete iterative schemes, 
dynamical systems assuming backward (implicit) evaluations of the governing operators have enjoyed much attention in the last years. 
Abbas and Attouch addressed in \cite{AbbasAttouch:2014} a forward-backward dynamical system associated to the solving of \eqref{monpr} for $A$ the convex
subdifferential of a proper, convex and lower semicontinuous function and $B$ a cocoercive operator, extending in this way the investigations
made by Bolte in \cite{Bolte:2003} on a gradient-projected dynamical system associated to the constrained minimization of a smooth convex function. 
The study in \cite{AbbasAttouch:2014} has been further extended in \cite{BotCsetnek:2015a}, 
this time for an arbitrary maximally monotone operator $A$ and also by utilizing variable relaxation parameters, a fact which permitted the derivation of convergence rates for 
the fixed point residual of the generated trajectories. Recently, in \cite{BotCsetnek:2015b}, the monotone inclusion problem \eqref{monpr} for $B$ cocoercive has been approached in terms of a 
second order dynamical system of forward-backward type with variable relaxation parameters and anisotropic damping/variable damping parameters 
(see also \cite{Antipin:1994, AttouchAlvarez:2000}). For more literature addressing dynamical systems of implicit type we refer the reader to \cite{AbbasAttouchSvaiter:2014, 
AttouchMarquesSvaiter:2015, AttouchSvaiter:2011, BotCsetnek:2015c}.

In the first part of the present manuscript we prove the existence of strong global solutions for the dynamical system \eqref{eq:CTseng} by making use the classical Cauchy--Lipschitz--Picard Theorem. This is followed by 
a convergence analysis for the generated trajectories. We show that  that $x(t)$ converges weakly, as $t \rightarrow +\infty$, to a solution of the monotone inclusion problem \eqref{monpr} under mild assumptions. 
We also show that, whenever $A + B$ is strongly monotone, the trajectories converge strongly with exponential rate. 

In the last part of the work we deal with the optimization problem
$$\mbox{minimize} \ f(x) + h(x),$$
where $f : \Hh \to \RRc$ is a proper, convex and lower semicontinuous function and $h: \Hh \to \RR$ is a convex differentiable one with Lipschitz continuous gradient, by taking into consideration that its set of minimizers 
is nothing else than the solution set of the monotone inclusion problem
\begin{equation*}
\mbox{find} \ \bar x \in \Hh \ \mbox{such that} \ 0 \in \partial f(\bar x) + \nabla h(\bar x). 
 \end{equation*}
We provide a rate of convergence for the objective function $f+h$ along the ergodic trajectories generated by \eqref{eq:CTseng} (for $A=\partial f$ and $B=\nabla h$) to its minimum value.

\section{Preliminaries}\label{sec2}
In this section we introduce some preliminary notions and recall some fundamental results that we will use throughout the paper. Let $\Hh$ be a real Hilbert space. 
A \emph{set-valued operator} $M: \Hh \toset \Hh$ maps points of $\Hh$ to subsets of $\Hh$. We denote by
\begin{align*}
  \Dom M &\defeq \setcond{x\in\Hh}{Mx \neq\emptyset}, \\
  \Ran M &\defeq \setcond{y\in\Hh}{\exists x\in \Hh: y\in Mx}, \\
  \Graph M &\defeq \setcond{\p{x, y}\in \Hh \times \Hh}{y \in Mx}, \\
  \zer M &\defeq \setcond{x\in\Hh}{0 \in Mx}
\end{align*}
its \emph{domain}, \emph{range}, \emph{graph} and \emph{zeros}, respectively. The \emph{inverse} operator of $M$ is defined by $M^{-1} y = \setcond{x\in \Hh}{y \in Mx}$, 
the \emph{multiplication by a scalar} $\lambda \in \RR$ by $\p{\lambda M}x = \setcond{\lambda y}{y\in Mx}$, and the \emph{sum} with another operator $B : \Hh \toset \Hh$ via Minkowski sums by
$\p{M + N}x = \setcond{a + b}{a \in Mx \text{ and }b\in Nx}$.

A set-valued operator $M: \Hh \toset \Hh$ is called \emph{monotone} if
\[
  \inpr{x - y}{x^* - y^*} \geq 0 \qquad \text{for all } x, y \in \Hh \text{ and } x^* \in Mx, y^* \in My.
\]
It is called \emph{maximally monotone} if it is monotone and there is no monotone operator whose graph contains $\Graph M$ properly. It is said to be \emph{$\rho$-strongly monotone} with $\rho>0$
\[
\langle x-y,u-v\rangle\geq \gamma\|x-y\|^2 \qquad \text{for all } x, y \in \Hh \text{ and } x^* \in Mx, y^* \in My.
\]
Notice that if $M$ is maximally monotone and strongly monotone, then $\zer M$ is a singleton, thus nonempty (see \cite[Corollary 23.37]{BauschkeCombettes:2011}).

The \emph{resolvent} $J_{\gamma M} = \p{\Id + \gamma M}^{-1}$ 
of the maximally monotone operator $\gamma M$ for $\gamma > 0$ is a single-valued operator with $\Dom J_{\gamma M} = \Hh$ and it is \emph{firmly nonexpansive}, i.e.,
\[
  \norm{J_{\gamma M} x - J_{\gamma M} y}^2 \leq \inpr{J_{\gamma M} x - J_{\gamma M} y}{x - y} \qquad \text{for all } x, y \in \Hh.
\]
Here, $\Id : \Hh \rightarrow \Hh$ denotes the \emph{identity operator} on $\Hh$. The \emph{Yosida approximation} of a maximally monotone operator $M$ with parameter $\gamma > 0$ is 
defined by $M_\gamma \defeq \frac{1}{\gamma} \p{\Id - J_{\gamma M}}$. It is $\frac{1}{\gamma}$-Lipschitz continuous, and it holds
\[
  x \in \zer M \iff J_{\gamma M} x = x \iff x \in \zer M_\gamma x.
\]
According to \cite[Proposition 23.28]{BauschkeCombettes:2011} we have the relation
\begin{equation}\label{ineqres}
  \norm{J_{\lambda M}x - J_{\mu M}x} \leq \abs{\lambda - \mu} \norm{M_{\lambda} x} \ \forall \lambda, \mu > 0 \ \forall x \in \Hh.
\end{equation}
Let $\beta>0$ be arbitrary. A single-valued operator $M: \Hh \rightarrow \Hh$ is said to be \emph{$\beta$-cocoercive}, if 
$\langle x-y,Mx-My\rangle\geq \beta\|Mx-My\|^2$ for all $(x,y)\in \Hh \times \Hh$, and \textit{$\frac{1}{\beta}$-Lipschitz continuous}, 
if $\|Mx-My\|\leq \frac{1}{\beta} \|x-y\|$ for all $(x,y)\in \Hh \times \Hh$. Obviously, every $\beta$-cocoercive operator is monotone and $\frac{1}{\beta}$-Lipschitz continuous, however, the opposite implication is not true.

A function $f: \Hh \to \RRc \defeq \RR \cup \set{\pm \infty}$ is said to be \emph{proper} if it does not take the value $-\infty$ and $\dom f \defeq \setcond{x\in \Hh}{f\p{x} < +\infty} \neq \emptyset$. It is called \emph{convex} if
\[
  f\p{\p{1 - \lambda}x + \lambda y} \leq \p{1 - \lambda} f\p{x} + \lambda f\p{y} \ \forall \lambda \in [0,1] \ \forall x, y \in \Hh.
\]
The \emph{conjugate function} $f^*: \Hh \to \RRc$ is defined by
\[
  f^*\p{x^*} = \sup\setcond{\inpr{x^*}{x} - f\p{x}}{x\in \Hh}
\]
and it is convex and lower semicontinuous. If $f$ is proper, convex and lower semicontinuous, then $f^*$ is also proper.

The \emph{convex subdifferential} of $f$ is defined by
\[
  \partial f \p{x} = \setcond{x^*\in \Hh}{\forall y\in \Hh: f\p{y} \geq f\p{x} + \inpr{x^*}{y - x}}
\]
for $f(x) \in \RR$ and $\partial f\p{x} = \emptyset$, otherwise. It is a set-valued monotone operator $\partial f: \Hh \toset \Hh$, which is maximally monotone if $f$ is proper, convex and lower semicontinuous.

We close this section by stating the solution concept we consider for the dynamical system \eqref{eq:CTseng}. 

\begin{definition}\label{abs-cont} \rm (see for instance \cite{AttouchSvaiter:2011, AbbasAttouchSvaiter:2014}) A function $x:[0,b]\rightarrow {\cal H}$ (where $b>0$) is said to be absolutely continuous if one of the 
following equivalent properties holds:

(i)  there exists an integrable function $y:[0,b]\rightarrow {\cal H}$ such that $$x(t)=x(0)+\int_0^t y(s)ds \ \ \forall t\in[0,b];$$

(ii) $x$ is continuous and its distributional derivative is Lebesgue integrable on $[0,b]$; 

(iii) for every $\varepsilon > 0$, there exists $\eta >0$ such that for any finite family of intervals $I_k=(a_k,b_k) \subseteq [0,b]$ we have the implication
$$\left(I_k\cap I_j=\emptyset \mbox{ and }\sum_k|b_k-a_k| < \eta\right)\Longrightarrow \sum_k\|x(b_k)-x(a_k)\| < \varepsilon.$$
\end{definition}

\begin{remark}\label{rem-abs-cont}\rm (a) It follows from the above definition that an absolutely continuous function on $[0,b]$ is differentiable almost 
everywhere, its derivative coincides with its distributional derivative almost everywhere and one can recover the function from its derivative $\dot x=y$ by the integration formula (i). 

(b) If $x:[0,b]\rightarrow {\cal H}$ (where $b>0$) is absolutely continuous and $M: \Hh \rightarrow \Hh$ is a $\gamma$-Lipschitz continuous operator for $\gamma > 0$, 
then the function $z=M\circ x$ is absolutely continuous, too. This follows from the characterization of absolute continuity given in
Definition \ref{abs-cont}(iii). Moreover, $z$ is almost everywhere differentiable and the inequality $\|\dot z (\cdot)\|\leq \gamma \|\dot x(\cdot)\|$ holds almost everywhere.   
\end{remark}

\begin{definition}\label{str-sol}\rm We say that $x:[0,+\infty)\rightarrow {\cal H}$ is a strong global solution of \eqref{eq:CTseng} if the 
following properties are satisfied: 

(i) $x:[0,+\infty)\rightarrow {\cal H}$ is \emph{locally absolutely continuous}, that is, absolutely continuous on each interval $[0,b]$ for $0<b<+\infty$; 

(ii) For almost every $t\in[0,+\infty)$ it holds $\dot x\p{t} + x\p{t} - z\p{t} - \gamma\p{t} Bx\p{t} + \gamma\p{t} Bz\p{t} = 0$, where $z\p{t} = J_{\gamma\p{t} A} \p{x\p{t} - \gamma\p{t} Bx\p{t}}$;

(iii) $x(0)=x_0$.
\end{definition}

\section{Existence and uniqueness of trajectories}\label{sec3}
In this section we investigate the existence and uniqueness of the trajectories generated by the dynamical system \eqref{eq:CTseng}. To this end we notice that the latter can be written 
as a non-autonomous differential equation
\[
  \dot x\p{t} = J_{\gamma\p{t}A}\p{x\p{t} - \gamma\p{t} Bx\p{t}} - x\p{t} + \gamma\p{t} Bx\p{t} - \gamma\p{t} \p{B \circ J_{\gamma\p{t} A}}\p{x\p{t} - \gamma\p{t} Bx\p{t}}
\]
or, equivalently, as 
$$\dot x\p{t} = f\p{\gamma\p{t}, x\p{t}},$$ 
with $f : (0,+\infty) \times \Hh \rightarrow \RR$,
\begin{align*}
  f\p{\gamma, x} \defeq& \, J_{\gamma A} \p{x - \gamma Bx} - x + \gamma Bx - \gamma \p{B \circ J_{\gamma A}}\p{x - \gamma Bx} \\
  =& \, \p{\p{\Id - \gamma B} \circ J_{\gamma A} \circ \p{\Id - \gamma B} - \p{\Id - \gamma B}}x.
\end{align*}

\begin{lemma}\label{lem1}
  Let $x\in \Hh$ be fixed. Then the function $\gamma \mapsto f\p{\gamma, x}$ is continuous on \opint{0}{+\infty}. Moreover, if $x \in \Dom A$,
  \[
    \lim_{\gamma \downarrow 0} f\p{\gamma, x} = 0.
  \]
\end{lemma}
\begin{proof}
  The first statement is a direct consequence of \eqref{ineqres}. Let $x\in \Dom A$. By nonexpansiveness of $J_{\gamma A}$ we have
  \[
    \norm{J_{\gamma A} \circ \p{\Id - \gamma B} x - J_{\gamma A} x} \leq \gamma\norm{Bx} \ \forall \gamma > 0.
  \]
On the other hand, $J_{\gamma A} x \to \Proj_{\cl\p{\Dom A}} x = x$ as $\gamma \to 0$ by \cite[Th\'eor\`eme 2.2]{Brezis:1973}, where $\Proj$ denotes the projection operator and one
uses that ${\cl\p{\Dom A}}$ is a convex and closed set. Hence $J_{\gamma A} \circ \p{\Id - \gamma B}x \to x$ as $\gamma \to 0$ and the assertion follows from the Lipschitz continuity of $B$.
\end{proof}

\begin{lemma}\label{lem:TsengLipschitz}
For each $\gamma \in (0,\beta)$ and $x, y\in \Hh$ it holds 
\[
 \norm{f\p{\gamma, x} - f\p{\gamma, y}} \leq \sqrt{6} \norm{x - y}.
\]
\end{lemma}
\begin{proof}
  For the sake of brevity, let us write $C \defeq \Id - \gamma B$ and $J \defeq J_{\gamma A}$. By using the firm nonexpansiveness of the resolvent and the monotonicity and 
  Lipschitz continuity of $B$ we get
  \begin{align*}
    &\mathrel{\phantom{=}} \norm{f\p{\gamma, x} - f\p{\gamma, y}}^2 \\
    &= \norm{C\circ J \circ Cx - Cx - C\circ J \circ Cy + Cy}^2 \\
    &= \norm{C\circ J \circ Cx - C\circ J \circ Cy}^2 + \norm{Cx - Cy}^2 - 2\inpr{C \circ J \circ Cx - C\circ J\circ Cy}{Cx - Cy} \\
    &= \norm{J\circ Cx - J\circ Cy}^2 + \gamma^2\norm{B \circ J \circ Cx - B\circ J \circ Cy}^2\\ 
    &\qquad \mathop{-} 2\gamma\inpr{J \circ Cx - J\circ Cy}{B \circ J \circ Cx - B\circ J \circ Cy} \\
    &\qquad \mathop{+} \norm{Cx - Cy}^2 - 2\inpr{C\circ J \circ Cx - C\circ J \circ Cy}{Cx - Cy} \\
    &\leq \p{1 + \frac{\gamma^2}{\beta^2}} \inpr{Cx - Cy}{J\circ Cx - J \circ Cy} \\
    &\qquad \mathop{-} 2\gamma\inpr{J \circ Cx - J\circ Cy}{B \circ J \circ Cx - B\circ J \circ Cy} \\
    &\qquad \mathop{+} \norm{Cx - Cy}^2 - 2\inpr{C\circ J \circ Cx - C\circ J \circ Cy}{Cx - Cy} \\
    &= \p{1 + \frac{\gamma^2}{\beta^2} - 2} \inpr{Cx - Cy}{J \circ Cx - J \circ Cy}  \\
    &\qquad \mathop{-} 2\gamma\inpr{J \circ Cx - J\circ Cy}{B \circ J \circ Cx - B\circ J \circ Cy} \\
    &\qquad \mathop{+} \norm{Cx - Cy}^2 + 2\gamma\inpr{B\circ J \circ Cx - B\circ J \circ Cy}{Cx - Cy} \\
    &\leq \norm{Cx - Cy}^2 + 2\gamma\norm{B\circ J \circ Cx - B\circ J \circ Cy} \norm{Cx - Cy} \\
    &\leq \p{1 + \frac{2\gamma}{\beta}} \norm{Cx - Cy}^2 \\
    &= \p{1 + \frac{2\gamma}{\beta}}\p{\norm{x - y}^2 + \gamma^2 \norm{Bx - By}^2 - 2\gamma \inpr{x - y}{Bx - By}} \\
    &\leq \p{1 + \frac{2\gamma}{\beta}}\p{1 + \frac{\gamma^2}{\beta^2}} \norm{x - y}^2 \\
    &\leq 6\norm{x - y}^2.
  \end{align*}
\end{proof}

\begin{lemma}\label{lem:linearestimate}
  There exists a constant $K > 0$ such that
  \[
    \norm{f\p{\gamma, x}} \leq K \p{1 + \norm{x}}
  \]
  for every $\gamma \in \opint{0}{\beta}$ and $x \in \Hh$.
\end{lemma}
\begin{proof}
We fix an element $\bar x \in \Dom A$ in the domain of $A$, which is evidently nonempty. According to Lemma \ref{lem1} the mapping $\gamma \mapsto f\p{\gamma, \bar x}$ can be continuously 
extended to $\gamma = 0$, therefore the image of $\clint{0}{\beta}$ under this extension is compact, hence bounded, say, $\norm{f\p{\gamma, \bar x}} \leq r$ for all $\gamma \in \opint{0}{\beta}$. Furthermore, by Lemma \ref{lem:TsengLipschitz} and the triangle inequaity,
  \begin{align*}
    \norm{f\p{\gamma, x}}
    &\leq \norm{f\p{\gamma, x} - f\p{\gamma, \bar x}} + \norm{f\p{\gamma, \bar x}} \\
    &\leq \sqrt{6} \norm{x - \bar x} + r \\
    &\leq \sqrt{6} \norm{\bar x} + r + \sqrt{6} \norm{x}. \qedhere
  \end{align*}
\end{proof}

Now we can state the existence and uniqueness statement. 
\begin{theorem}\label{thexun}
  Let $\gamma: \coint{0}{+\infty} \to \opint{0}{\beta}$ be measurable. 
 Then, for each $x_0 \in \Hh$, there exists a unique function $x: \coint{0}{+\infty} \to \Hh$ with $x\p{0} = x_0$, 
  which is locally absolutely continuous and $\dot x\p{t} = f\p{\gamma\p{t}, x\p{t}}$ for almost every $t \in [0,+\infty)$.
\end{theorem}
\begin{proof}
The statement follows as a consequence of the Cauchy--Lipschitz--Picard Theorem (see \cite[Proposition 6.2.1]{Haraux:1991}) applied for the mapping $(t,x) \mapsto f(\gamma(t),x)$ under the use of the previous two lemmas. 
For arbitrary $x,y \in \Hh$ and every $t \in [0,+\infty)$, by Lemma \ref{lem:TsengLipschitz}, we have
$$\|f(\gamma(t),x) - f(\gamma(t),y)\| \leq  \sqrt{6} \norm{x - y}.$$
On the other hand, we recall that $\gamma \mapsto f\p{\gamma, x}$ is continuous on $(0,+\infty)$ for each $x\in \Hh$, so $t \mapsto f\p{\gamma\p{t}, x}$ is measurable, 
and it is bounded by Lemma \ref{lem:linearestimate}, thus locally integrable.
\end{proof}

\section{Convergence analysis}\label{sec4}

In order to investigate the asymptotic properties of \eqref{eq:CTseng} we need some inequalities which we derive in the next subsection.

\subsection{Some fundamental inequalities}\label{subsec41}

\begin{lemma}\label{lem:inclusions} If $x$ and $z$ are given by \eqref{eq:CTseng}, then, for almost every $t \in [0,+\infty)$, the following statements are true:
  \begin{enumerate}[label=(\alph*)]
    \item \label{item:lem:inclusions:A} $\frac{x\p{t} - z\p{t}}{\gamma\p{t}} - Bx\p{t} \in Az\p{t}$;
    \item \label{item:lem:inclusions:A+B} $\frac{x\p{t} - z\p{t}}{\gamma\p{t}} + Bz\p{t} - Bx\p{t} = -\frac{\dot x\p{t}}{\gamma\p{t}} \in \p{A + B} z\p{t}$.
  \end{enumerate}
\end{lemma}
\begin{proof}
The statement in \ref{item:lem:inclusions:A} is a reformulation of the first equation in \eqref{eq:CTseng}, while the one in \ref{item:lem:inclusions:A+B} follows by adding $Bz\p{t}$ to \ref{item:lem:inclusions:A}
and by using the second equation in \eqref{eq:CTseng}.
\end{proof}

\begin{lemma}\label{lemma5}
Let $x$ and $z$ be given by \eqref{eq:CTseng} and $\bar x \in \zer\p{A + B}$. Then, for almost every $t\in \coint{0}{+\infty}$,  we have
  \[
    0 \leq \norm{x\p{t} - \bar x}^2 - \norm{x\p{t} - z\p{t}}^2 - \norm{z\p{t} - \bar x}^2 + 2\gamma\p{t} \inpr{B\bar x - Bx\p{t}}{z\p{t} - \bar x}.
  \]
\end{lemma}
\begin{proof}
 As $-B\bar x \in A\bar x$, it holds $-\gamma\p{t} B\bar x \in \gamma\p{t} A\bar x$ for every $t \in [0, +\infty)$. By Lemma \ref{lem:inclusions}\ref{item:lem:inclusions:A} and the monotonicity of $A$, 
 for almost every $t\in \coint{0}{+\infty}$ it holds
  \begin{align*}
    0 &\leq 2\inpr{x\p{t} - \gamma\p{t} Bx\p{t} - z\p{t} + \gamma\p{t} B\bar x}{z\p{t} - \bar x} \\
    &= \norm{x\p{t} - \bar x}^2 - \norm{x\p{t} - z\p{t}}^2 - \norm{z\p{t} - \bar x}^2 + 2\gamma\p{t} \inpr{B\bar x - Bx\p{t}}{z\p{t} - \bar x}. \qedhere
  \end{align*}
\end{proof}

\begin{lemma}\label{lem:zderiv}
Let $x$ and $z$ be given by \eqref{eq:CTseng}, and let $\gamma:[0,+\infty) \rightarrow (0,\beta)$ be locally absolutely continuous. Then $z$ is locally absolutely continuous, and
  \[
    \norm{\dot z\p{t}} \leq \p{\sqrt{\p{1 + \frac{\dot \gamma\p{t}}{\gamma\p{t}}}^2 + \frac{\p{\gamma\p{t}}^2}{\beta^2}} + \frac{\gamma\p{t}}{\beta} \sqrt{1 + \frac{\p{\gamma\p{t}}^2}{\beta^2}}} \norm{x\p{t} - z\p{t}} \qedhere
  \]
for almost every $t\in \coint{0}{+\infty}$.
\end{lemma}
\begin{proof}
Let $b >0$ be fixed. Since $x, Bx$ and $\gamma$ are absolutely continuous on $[0,b]$, the mapping $t \mapsto y(t):=x(t) -  \gamma(t)Bx(t)$ is absolutely continuous on $[0,b]$. 
We show that $t \mapsto J_{\gamma(t)A}(y(t))$ is absolutely continuous on $[0,b]$, as well.

For every $s, t \in [0, b]$, by using \eqref{ineqres} and the nonexpansiveness of the resolvent, we get
\begin{align*}
\norm{z(t) - z(s)} = &  \norm{J_{\gamma(t)A}(y(t)) - J_{\gamma(s)A}(y(s))}\\
  = &  \norm{J_{\gamma(t)A}(y(t)) - J_{\gamma(t)A}(y(s))} + \norm{J_{\gamma(t)A}(y(s)) - J_{\gamma(s)A}(y(s))}\\
\leq &  \norm{y(t) - y(s)} + |\gamma(t) - \gamma(s)|\norm{A_{\gamma(t)}(y(t))}.
\end{align*}
Since $\gamma$ is continuous on $[0,b]$, there exist $\gamma_{\min}, \gamma_{\max} \in (0,\beta)$ such that $\gamma_{\min} \leq \gamma(\cdot) \leq \gamma_{\max}$ on $[0,b]$. 
Using that $\gamma \mapsto \|A_{\gamma}(y(t))\|$ is nonincreasing  and the Lipschitz continuity of the Yosida approximation, it yields for every $s, t \in [0, b]$
\begin{align*}
\norm{z(t) - z(s)} \leq &  \norm{y(t) - y(s)} + |\gamma(t) - \gamma(s)|\norm{A_{\gamma(t)}(y(t))}\\
\leq &  \norm{y(t) - y(s)} + |\gamma(t) - \gamma(s)|\norm{A_{\gamma_{\min}}(y(t))}\\
\leq &  \norm{y(t) - y(s)} + |\gamma(t) - \gamma(s)|\left (\norm{A_{\gamma_{\min}}(0)} + \frac{1}{\gamma_{\min}}\norm{y(t)} \right).
\end{align*}
From here the absolute continuity of $z$ on $[0, b]$ follows, by taking into consideration also that $y$ is bounded.

Applying Lemma \ref{lem:inclusions} \ref{item:lem:inclusions:A} for $s, t \in [0, b], s \neq t, $ we obtain by the monotonicity of $B$
  \begin{align*}
    0 &\leq \inpr{z\p{s} - z\p{t}}{\frac{x\p{s} - z\p{s}}{\gamma\p{s}} - Bx\p{s} - \frac{x\p{t} - z\p{t}}{\gamma\p{t}} + Bx\p{t}},
  \end{align*}
  which is equivalent to
  \begin{align*}
& \norm{\frac{z\p{s} - z\p{t}}{s - t}}^2 \leq\\ 
& \inpr{\frac{z\p{s} - z\p{t}}{s - t}}{\frac{x\p{s} - x\p{t}}{s - t} + \frac{\gamma\p{t} - \gamma\p{s}}{s - t} \cdot \frac{x\p{t} - z\p{t}}{\gamma\p{t}} + \gamma\p{s} \frac{Bx\p{t} - Bx\p{s}}{s - t}},
  \end{align*}
  so, by the Cauchy--Schwarz inequality,
  \[
    \norm{\frac{z\p{s} - z\p{t}}{s - t}} \leq \norm{\frac{x\p{s} - x\p{t}}{s - t} + \frac{\gamma\p{t} - \gamma\p{s}}{s - t} \cdot \frac{x\p{t} - z\p{t}}{\gamma\p{t}} + \gamma\p{s} \cdot \frac{Bx\p{t} - Bx\p{s}}{s - t}}.
  \]
By taking the limit $s \to t$, it follows that for almost every $t \in [0,+\infty)$
  \begin{align}
    \norm{\dot z\p{t}} &\leq \norm{\dot x\p{t} - \frac{\dot \gamma\p{t}}{\gamma\p{t}} \p{x\p{t} - z\p{t}} - \gamma\p{t} \frac{\dd{}}{\dd t} Bx\p{t}} \nonumber \\
    &= \norm{\p{1 + \frac{\dot \gamma\p{t}}{\gamma\p{t}}} \p{z\p{t} - x\p{t}} + \gamma\p{t} \p{Bx\p{t} - Bz\p{t}} - \gamma\p{t} \frac{\dd{}}{\dd t} Bx\p{t}}. \label{eq:deriv_estimation}
  \end{align}
According to Remark \ref{rem-abs-cont}(b) we have $\norm{\frac{\dd{}}{\dd t} Bx\p{t}} \leq \frac{1}{\beta} \norm{\dot x\p{t}}$ for almost every $t \in [0,+\infty)$.
Furthermore, by the monotonicity and the Lipschitz continuity of $B$, we have for almost every $t \in [0,+\infty)$
  \begin{align}\label{xprime}
    \norm{\dot x\p{t}}^2 &= \norm{x\p{t} - z\p{t}}^2 + \p{\gamma\p{t}}^2 \norm{Bx\p{t} - Bz\p{t}}^2 + 2\gamma\p{t} \inpr{x\p{t} - z\p{t}}{Bz\p{t} - Bx\p{t}} \nonumber \\
    &\leq \p{1 + \frac{\p{\gamma\p{t}}^2}{\beta^2}} \norm{x\p{t} - z\p{t}}^2
  \end{align}
  as well as
  \begin{align*}
    &\mathrel{\phantom{=}} \norm{\p{1 + \frac{\dot\gamma\p{t}}{\gamma\p{t}}}\p{z\p{t} - x\p{t}} + \gamma\p{t}\p{Bx\p{t} - Bz\p{t}}}^2 \\
    &\leq \p{\p{1 + \frac{\dot \gamma\p{t}}{\gamma\p{t}}}^2 + \frac{\p{\gamma\p{t}}^2}{\beta^2}} \norm{x\p{t} - z\p{t}}^2
  \end{align*}
  so, getting back to \eqref{eq:deriv_estimation}, we obtain
  \[
    \norm{\dot z\p{t}} \leq \p{\sqrt{\p{1 + \frac{\dot \gamma\p{t}}{\gamma\p{t}}}^2 + \frac{\p{\gamma\p{t}}^2}{\beta^2}} + \frac{\gamma\p{t}}{\beta} \sqrt{1 + \frac{\p{\gamma\p{t}}^2}{\beta^2}}} \norm{x\p{t} - z\p{t}}. \qedhere
  \]
\end{proof}

When $A+B$ is strongly monotone, we have the following strengthened version of the inequality in Lemma \ref{lemma5}.

\begin{lemma}\label{lem:strongmon}
Let $A + B$ be $\rho$-strongly monotone for  $\rho > 0$, $x$ and $z$ be given by \eqref{eq:CTseng} and $\bar x \in \zer\p{A + B}$. Then for almost every $t\in \coint{0}{+\infty}$  we have
  \begin{multline*}
    0 \leq \norm{x\p{t} - \bar x}^2 - \norm{x\p{t} - z\p{t}}^2 - \p{1 + 2\rho\gamma\p{t}} \norm{z\p{t} - \bar x}^2 \\
    \mathop{+} 2\gamma\p{t} \inpr{Bz\p{t} - Bx\p{t}}{z\p{t} - \bar x}.
  \end{multline*}
\end{lemma}
\begin{proof}
As $0 \in \gamma(t)\p{A + B} \bar x$, and $\p{A + B}$ is $\rho$-strongly monotone, by taking Lemma \ref{lem:inclusions} \ref{item:lem:inclusions:A+B} into consideration, we have for almost every $t\in \coint{0}{+\infty}$
  \[
    2\rho\gamma\p{t} \norm{z\p{t} - \bar x}^2 \leq 2\inpr{x\p{t} - z\p{t} + \gamma\p{t} Bz\p{t} - \gamma\p{t} Bx\p{t}}{z\p{t} - \bar x}
  \]
  and the assertion follows by rearranging the terms.
\end{proof}

\subsection{Asymptotic properties of the trajectories}\label{subsec42}

The following result, for the proof of which we refer to \cite[Lemma 5.2]{AbbasAttouchSvaiter:2014}, is the continuous counterpart of a classical result which states the convergence of 
quasi-Fej\'er monotone sequences. 

\begin{lemma}\label{fejer-cont2}  If $1 \leq p < \infty$, $1 \leq r \leq \infty$, $F:[0,+\infty)\rightarrow[0,+\infty)$ is 
locally absolutely continuous, $F\in L^p([0,+\infty))$, $G:[0,+\infty)\rightarrow\RR$, $G\in  L^r([0,+\infty))$ and 
for almost every $t \in [0,+\infty)$ $$\frac{\dd {}}{\dd t}F(t)\leq G(t),$$ then $\lim_{t\rightarrow +\infty} F(t)=0$. 
\end{lemma}

The next result which we recall here is the continuous version of the Opial Lemma (see, for example, \cite[Lemma 5.3]{AbbasAttouchSvaiter:2014}, \cite[Lemma 1.10]{AbbasAttouch:2014}). 

\begin{lemma}\label{opial} Let $S \subseteq {\cal H}$ be a nonempty set and $x:[0,+\infty)\rightarrow{\cal H}$ a given map. Assume that 

(i) for every $\bar x \in S$, $\lim_{t\rightarrow+\infty}\|x(t)-\bar x\|$ exists; 

(ii) every weak sequential cluster point of the map $x$ belongs to $S$. 

\noindent Then there exists $x_{\infty}\in S$ such that  $x(t)$ converges weakly to $x_{\infty}$ as $t \rightarrow +\infty$. 
\end{lemma}

The following proposition will play an essential role when establishing the asymptotic properties of the trajectories generated by \eqref{eq:CTseng}.

\begin{lemma}\label{prop:Fejermonotonicity}
  Let $\bar x \in \zer\p{A + B}$. Then $t \mapsto \norm{x\p{t} - \bar x}$ is monotonically decreasing and $\int_0^{+\infty} \p{1 - \frac{\gamma\p{t}}{\beta}} \norm{x\p{t} - z\p{t}}^2 \dd t< +\infty$.
\end{lemma}
\begin{proof}
For almost every $t\in \coint{0}{+\infty}$, by using Lemma \ref{lemma5}, the monotonicity and the Lipschitz continuity of $B$, we have
\begingroup
\allowdisplaybreaks
\begin{align*}
 \frac{\dd{}}{\dd t} \norm{x\p{t} - \bar x}^2 = & \ 2\inpr{x\p{t} - \bar x}{\dot x\p{t}} \\
  = & \ 2\inpr{x\p{t} - \bar x}{z\p{t} - x\p{t} + \gamma\p{t} Bx\p{t} - \gamma\p{t} Bz\p{t}} \\
  = & \ \norm{z\p{t} - \bar x}^2 - \norm{x\p{t} - z\p{t}}^2 - \norm{x\p{t} - \bar x}^2\\ 
     & \ + 2\gamma\p{t} \inpr{x\p{t} - \bar x}{Bx\p{t} - Bz\p{t}} \\
\leq & \ - 2\norm{x\p{t} - z\p{t}}^2 + 2\gamma\p{t} \inpr{B\bar x - Bx\p{t}}{z\p{t} - \bar x}\\ 
& \ + 2\gamma\p{t} \inpr{x\p{t} - \bar x}{Bx\p{t} - Bz\p{t}} \\
\leq & \ - 2\norm{x\p{t} - z\p{t}}^2 + 2\gamma\p{t} \inpr{Bz\p{t} - Bx\p{t}}{z\p{t} - \bar x}\\ 
& \ + 2\gamma\p{t} \inpr{x\p{t} - \bar x}{Bx\p{t} - Bz\p{t}} \\
 = & \ -2\norm{x\p{t} - z\p{t}}^2 + 2\gamma\p{t} \inpr{Bz\p{t} - Bx\p{t}}{z\p{t} - x\p{t}} \\
 \leq & \ 2\p{\frac{\gamma\p{t}}{\beta} - 1} \norm{x\p{t} - z\p{t}}^2 \leq 0,
\end{align*}
\endgroup
which shows the decreasing property. Integrating from $0$ to $T$, for $T >0$, yields
\[
  \int_0^T \p{1 - \frac{\gamma\p{t}}{\beta}} \norm{x\p{t} - z\p{t}}^2 \dd t \leq \frac{\norm{x\p{0} - \bar x}^2 - \norm{x\p{T} - \bar x}^2}{2} \leq \frac{\norm{x\p{0} - \bar x}^2}{2},
\]
which is independent of $T$.
\end{proof}

\begin{theorem}\label{th2}
Let $\zer\p{A + B} \neq \emptyset$ and let $\gamma$ be locally absolutely continuous such that, for some $\delta, \varepsilon > 0$, we have $\delta \leq \gamma\p{t} \leq \beta - \varepsilon$ for all $t\in [0, +\infty)$ and 
$\dot \gamma \in L^\infty([0, +\infty))$. Then the trajectories $x\p{t}$ and $z\p{t}$ generated by \eqref{eq:CTseng} converge weakly to an element in $\zer\p{A + B}$ as $t \rightarrow +\infty$.
\end{theorem}
\begin{proof}
According to Lemma \ref{prop:Fejermonotonicity} we have that $t \rightarrow \norm{x\p{t} - z\p{t}}^2$, mapping from $[0,+\infty)$ to $[0,+\infty)$, belongs to $L^1{\coint{0}{+\infty}}$.
Furthermore, by the Cauchy--Schwarz inequality, the triangle inequality, \eqref{xprime} and Lemma \ref{lem:zderiv} we have that for almost every $t \in [0,+\infty)$
  \begin{align*}
    \frac{\dd{}}{\dd t} \norm{x\p{t} - z\p{t}}^2 &= 2\inpr{x\p{t} - z\p{t}}{\dot x\p{t} - \dot z\p{t}} \\
    &\leq 2\p{\norm{\dot x\p{t}} + \norm{\dot z\p{t}}} \norm{x\p{t} - z\p{t}} \\
    &\leq \!\p{\!\!\sqrt{\p{1 + \frac{\dot \gamma\p{t}}{\gamma\p{t}}}^2 + \frac{\p{\gamma\p{t}}^2}{\beta^2}} \! + \!\p{1 + \frac{\gamma\p{t}}{\beta}} \sqrt{1 + \frac{\p{\gamma\p{t}}^2}{\beta^2}}}\!\norm{x\p{t} - z\p{t}}^2 \\
    &\leq \p{\sqrt{\p{1 + \frac{\norm{\dot \gamma}_{L^\infty([0, +\infty))}}{\delta}}^2 + 1} + 2\sqrt{2}} \norm{x\p{t} - z\p{t}}^2
  \end{align*}
By Lemma \ref{fejer-cont2} we have $\lim_{t\to +\infty} \norm{x\p{t} - z\p{t}}^2 = 0$, which implies, via \eqref{xprime}, that  $\dot x\p{t} \to 0$ as $t\to +\infty$. 
Let $w \in \Hh$ be a weak sequential cluster point of $x\p{t}$ as $t\to +\infty$ and $\p{t_n}_{n \geq 0}$ be a sequence in $\coint{0}{+\infty}$ with $t_n \to +\infty$ and $x\p{t_n} \weakto w$ as $n \to +\infty$. 
Since $\lim_{t\to +\infty} \p{x\p{t} - z\p{t}} = 0$, we also have $z\p{t_n} \weakto w$ as $n\to \infty$. Furthermore, $-\frac{\dot x\p{t_n}}{\gamma\p{t_n}} \to 0$ as $n \to +\infty$, since $\gamma\p{t_n} \geq \delta$ 
for all $n\geq 0$. 

By Lemma \ref{lem:inclusions} \ref{item:lem:inclusions:A+B} and the fact that the graph of the maximally monotone operator $A+B$ is sequentially weak-strong closed (see \cite[Corollary 24.4, Proposition 20.33]{BauschkeCombettes:2011}), 
we have $\p{w, 0} \in \Graph\p{A + B}$, thus $w \in \zer\p{A + B}$. By Lemma \ref{prop:Fejermonotonicity}, $\norm{x\p{t} - \bar x}$ converges as $t\to +\infty$. 
According to the Opial Lemma, $x\p{t}$ (and, consequently, $z\p{t}$) converges weakly to an element of $\zer\p{A + B}$ as $t\to +\infty$.
\end{proof}

For the important special case of strongly monotone inclusions, we are able to show strong convergence of the trajectories to solutions without any continuity assumptions on the function $\gamma$.
\begin{theorem}\label{th3}
  Let $A + B$ be $\rho$-strongly monotone for $\rho > 0$. and let $\bar x \in \zer\p{A + B}$. Then we have for every $t \in [0,+\infty)$ the estimate
  \[
    \norm{x\p{t} - \bar x}^2 \leq \frac{\norm{x\p{0} - \bar x}^2}{\exp\p{\int_0^t \frac{2\rho\gamma\p{s}\p{\beta - \gamma\p{s}}}{\beta\rho\gamma\p{s} + \beta - \gamma\p{s}} \dd s}}.
  \]
  In particular, if
  \[
    \int_0^{+\infty} \frac{2\rho\gamma\p{s} \p{\beta - \gamma\p{s}}}{\beta\rho\gamma\p{s} + \beta - \gamma\p{s}} \dd s = +\infty,
  \]
 then $x\p{t}$ converges in norm to the unique element of $\zer\p{A + B}$ as $t \rightarrow +\infty$.
\end{theorem}
\begin{proof}
For almost every $t\in \coint{0}{+\infty}$, by using Lemma \ref{lem:strongmon}, the monotonicity and the Lipschitz continuity of $B$, we have
\begin{align*}
    \frac{\dd{}}{\dd t} \norm{x\p{t} - \bar x}^2
    = & 2\inpr{x\p{t} - \bar x}{\dot x\p{t}} \\
    = & 2\inpr{x\p{t} - \bar x}{z\p{t} - x\p{t}} + 2\gamma\p{t} \inpr{x\p{t} - \bar x}{Bx\p{t} - Bz\p{t}} \\
    = & \norm{z\p{t} - \bar x}^2 - \norm{x\p{t} - z\p{t}}^2 - \norm{x\p{t} - \bar x}^2 \\
    & \mathop{+} 2\gamma\p{t} \inpr{x\p{t} - \bar x}{Bx\p{t} - Bz\p{t}} \\
    \leq & - 2\norm{x\p{t} - z\p{t}}^2 + 2\gamma\p{t} \inpr{x\p{t} - z\p{t}}{Bx\p{t} - Bz\p{t}} \\
    & \mathop{-} 2\rho\gamma\p{t} \norm{z\p{t} - \bar x}^2 \\
    \leq & -2\p{1 - \frac{\gamma\p{t}}{\beta}} \norm{x\p{t} - z\p{t}}^2 - 2\rho\gamma\p{t} \norm{z\p{t} - \bar x}^2.
  \end{align*}
For $\alpha: \coint{0}{+\infty} \to \opint{0}{+\infty}$, 
  \[
    \alpha\p{t} = 1 + \frac{\beta - \gamma\p{t}}{\beta\rho \gamma\p{t}} > 1,
  \]
we have 
\[
    \norm{z\p{t} - \bar x}^2 \geq \p{1 - \alpha\p{t}}\norm{x\p{t} - z\p{t}}^2 + \p{1 - \frac{1}{\alpha\p{t}}} \norm{x\p{t} - \bar x}^2,
  \]  
in other words,  
 \begin{align*}
    -2\rho\gamma\p{t} \norm{z\p{t} - \bar x}^2 \leq 2 \p{1 - \frac{\gamma\p{t}}{\beta}} \norm{x\p{t} - z\p{t}}^2 - 
    \frac{2\rho\gamma\p{t}\p{\beta - \gamma\p{t}}}{\beta\rho\gamma\p{t} + \beta - \gamma\p{t}} \norm{x\p{t} - \bar x}^2
 \end{align*}
for every $t \in [0,+\infty)$. Consequently,  
  \[
    \frac{\dd{}}{\dd t} \norm{x\p{t} - \bar x}^2 \leq - \frac{2\rho\gamma\p{t} \p{\beta - \gamma\p{t}}}{\beta \rho \gamma\p{t} + \beta - \gamma\p{t}} \norm{x\p{t} - \bar x}^2
  \]
for amost every $t \in [0,+\infty)$. By Gr\"onwall's inequality, for every $t \in [0,+\infty)$ we have
  \[
    \norm{x\p{t} - \bar x}^2 \leq \frac{\norm{x\p{0} - \bar x}^2}{\exp\p{\int_0^t \frac{2\rho\gamma\p{s}\p{\beta - \gamma\p{s}}}{\beta\rho\gamma\p{s} + \beta - \gamma\p{s}} \dd s}}. \qedhere
  \]
 \end{proof}

\begin{corollary}\label{cor:exponential}
 Let $A + B$ be $\rho$-strongly monotone for $\rho >0$ and $\gamma$ let be such that for $\delta, \varepsilon > 0$ we have $\delta \leq \gamma\p{t} \leq \beta - \varepsilon$ for all $t\in [0, +\infty)$
Then the trajectory $x\p{t}$ converges to the unique element of $\zer\p{A + B}$ with an exponential rate as $t \rightarrow +\infty$.
\end{corollary}
\begin{proof}
Let $\bar x$ be the unique element of $\zer\p{A + B}$. According to Theorem \ref{th3} it holds for every $t \in [0,+\infty)$
  \begin{align*}
    \norm{x\p{t} - \bar x}^2 &\leq \frac{\norm{x\p{0} - \bar x}^2}{\exp\p{\int_0^t \frac{2\rho\gamma\p{s}\p{\beta - \gamma\p{s}}}{\beta\rho\gamma\p{s} + \beta - \gamma\p{s}} \dd s}} 
    \leq \frac{\norm{x\p{0} - \bar x}^2}{\exp\p{\int_0^t \frac{2\rho\delta\varepsilon}{\beta\rho\p{\beta - \varepsilon} + \beta - \delta} \dd s}} \\
    &= \norm{x\p{0} - \bar x}^2 \ee^{-\frac{2\rho\delta\varepsilon t}{\beta\rho\p{\beta - \varepsilon} + \beta - \delta}}, 
  \end{align*}
which leads to the desired conclusion.
\end{proof}
\begin{remark}
  Corollary \ref{cor:exponential} can be seen as the continuous-time counterpart of \cite[Theorem 3.4]{BotHendrich:2014}.
\end{remark}

\subsection{Ergodic objective rate for convex minimization problems}\label{subsec43}

Consider the convex minimization problem
$$\mbox{minimize} \ f(x) + h(x),$$
where $f : \Hh \to \RRc$ is a proper, convex and lower semicontinuous function and $h: \Hh \to \RR$ a convex and Fr\'echet differentiable one with a $\frac{1}{\beta}$-Lipschitz continuous gradient for $\beta > 0$. Since
$\argmin (f+h) = \zer (\partial f + \nabla h)$, one can approach this set by means of the trajectories of the dynamical system \eqref{eq:CTseng} written for $A=\partial f$ and $B = \nabla h$. We notice that, for $\eta > 0$, 
the resolvent of $\eta \partial f$ is given by $J_{\eta \partial f}=\Prox_{\eta f}$ (see \cite{BauschkeCombettes:2011}),
where $\Prox_{\eta f}:{\cal H}\rightarrow {\cal H}$,
\begin{equation*}
\Prox\nolimits_{\eta f}(x)=\argmin_{y\in {\cal H}}\left \{f(y)+\frac{1}{2\eta}\|y-x\|^2\right\},
\end{equation*}
denotes the proximal point operator of $\eta f$. Thus, the dynamical system \eqref{eq:CTseng} becomes
\begin{equation}\label{eq:CTsengf}\left\{
\begin{array}{rl}
  z\p{t} \!\!\!&= \Prox_{\gamma\p{t} f} \p{x\p{t} - \gamma\p{t} \nabla h\p{x\p{t}}}, \\
  0 \!\!\!&= \dot x\p{t} + x\p{t} - z\p{t} - \gamma\p{t} \nabla h\p{x\p{t}} + \gamma\p{t} \nabla h\p{z\p{t}}\\
  x\p{0} \!\!\!&= x_0.
\end{array}\right.  
\end{equation}
In the following. we are concerned with the asymptotic behavior of the ergodic trajectory
\[
  \zeta\p{t} \defeq \frac{1}{\Gamma\p{t}} \int_0^t \gamma\p{s} z\p{s} \dd s, \qquad \text{where} \qquad \Gamma\p{t} \defeq \int_0^t \gamma\p{s} \dd s,
\]
of $z$ with weight $\gamma$ as $t \rightarrow +\infty$.

\begin{theorem}\label{thm:ergodic_objective_rate}
  Let $f: \Hh \to \RRc$ be proper, convex and lower semicontinuous, $h: \Hh \to \RR$ be convex and dand Fr\'echet differentiable one with a $\frac{1}{\beta}$-Lipschitz continuous gradient for $\beta > 0$, 
$x$ and $z$ be given by \eqref{eq:CTsengf} and let $\gamma:[0,+\infty) \rightarrow (0,\beta)$ be Lebesgue measurable. Then
 \begin{equation}\label{eqphi}
    \p{f + h}\p{\zeta\p{t}} \leq \p{f + h}\p{x} + \frac{\norm{x\p{0} - x}^2}{2\Gamma\p{t}}
\end{equation}
 for every $x \in \Hh$ and every $t > 0$ such that $\zeta\p{t} \in \dom f$.
\end{theorem}
\begin{proof}
Let $x \in \Hh$ and $t > 0$ be such that $\zeta\p{t} \in \dom f$ fixed. We have, by Lemma \ref{lem:inclusions} \ref{item:lem:inclusions:A+B},
  \[
    -\frac{\dot x\p{s}}{\gamma\p{s}} \in \p{\partial f + \nabla h}\p{z\p{s}} = \partial\p{f + h}\p{z\p{s}}
  \]
and further, by the subdifferential inequality, we obtain
  \begin{equation}\label{eq:CTseng_subdiff_ineq}
    \p{f + h}\p{x} \geq \p{f + h}\p{z\p{s}} + \inpr{-\frac{\dot x\p{s}}{\gamma\p{s}}}{x - z\p{s}}
  \end{equation}
 for almost every $s \in \coint{0}{+\infty}$.  
By the Lipschitz continuity of $B$ and the Cauchy--Schwarz inequality it follows that for almost every $s \in [0,+\infty)$
 \begin{align*}
\inpr{\dot x\p{s}}{x - z\p{s}}  = & \inpr{x\p{s} - z\p{s}}{\dot x\p{s}} + \inpr{x - x\p{s}}{\dot x\p{s}} \\
    = & \inpr{x\p{s} - z\p{s}}{z\p{s} - x\p{s} + \gamma\p{s} Bx\p{s} - \gamma\p{s} Bz\p{s}} + \inpr{x - x\p{s}}{\dot x\p{s}}\\
    \leq & -\norm{x\p{s} - z\p{s}}^2 + \frac{\gamma\p{s}}{\beta} \norm{x\p{s} - z\p{s}}^2 + \inpr{x - x\p{s}}{\dot x\p{s}}\\
    \leq & \inpr{x - x\p{s}}{\dot x\p{s}}\\ 
    = & -\frac{1}{2} \frac{\dd{}}{\dd s} \norm{x\p{s} - x}^2.
  \end{align*} 
 
 By Jensen's inequality in integral form, \eqref{eq:CTseng_subdiff_ineq} and the previous estimate we have
  \begin{align*}
    \p{f + h}\p{\zeta\p{t}}
    &= \p{f + h}\p{\frac{1}{\Gamma\p{t}} \int_0^{t} \gamma\p{s} z\p{s} \dd s} \\
    &\leq \frac{1}{\Gamma\p{t}} \int_0^t \gamma\p{s} \p{f + h}\p{z\p{s}} \dd s \\
    &\leq \frac{1}{\Gamma\p{t}} \int_0^t \gamma\p{s} \p{\p{f + h}\p{x} + \frac{1}{\gamma\p{s}} \inpr{\dot x\p{s}}{x - z\p{s}}} \dd s \\
    &\leq \p{f + h}\p{x} + \frac{1}{\Gamma\p{t}} \int_0^t -\frac{1}{2} \frac{\dd{}}{\dd s} \norm{x\p{s} - x}^2 \dd s \\
    &= \p{f + h}\p{x} + \frac{\norm{x\p{0} - x}^2 - \norm{x\p{t} - x}^2}{2\Gamma\p{t}},
  \end{align*}
  from which the assertion follows by neglecting the nonpositive term $\frac{-\norm{x\p{t} - x}^2}{2\Gamma\p{t}}$.
\end{proof}

\begin{remark}\label{jensen}
If $\Hh$ is finite dimensional or $\dom f$ is closed, then one can use Jensen's inequality in integral form in a less restrictive way (see for instance \cite{Perlman:1974}). Under these premises inequality
\eqref{eqphi} in Theorem \ref{thm:ergodic_objective_rate} is fulfilled for every $x \in \Hh$ and every $t > 0$.
\end{remark}

\begin{remark}
  The statement of Theorem \ref{thm:ergodic_objective_rate} holds in a similar form for the discrete version given by the forward-backward-forward iterative algorithm, too. Let $x_0 \in \Hh$ be arbitrary, $(\gamma_n)_{n\geq 0} \subseteq (0,\beta)$, 
  and $\p{x_n}_{n\geq 0}$ and $\p{z_n}_{n\geq 0}$ be the sequences generated by
\begin{equation*}
(\forall n \geq 0) \ \left\{
\begin{aligned}
  z_n &:= \Prox_{\gamma_n f} \p{x_n - \gamma_n \nabla h(x_n)}  \\
  x_{n+1} &:= z_n + \gamma_n \p{\nabla h(x_n) - \nabla h(z_n)}.
\end{aligned}
\right.
\end{equation*}  
Let $x \in \Hh$ and $n \geq 0$ be fixed. For any $k \geq 0$ we have
$$x_k - z_k - \gamma_k \nabla h\p{x_k} \in \gamma_k \partial f\p{z_k},$$ so
  \[
    \frac{x_k - z_k}{\gamma_k} - \nabla h\p{x_k} + \nabla h\p{z_k} \in \partial\p{f + h}\p{z_k}.
  \]
The subgradient inequality yields for any $k \geq 0$
  \begin{align*}
    \p{f + h}\p{z_k} &\leq \p{f + h}\p{x} - \inpr{\frac{x_k - z_k}{\gamma_k} - \nabla h\p{x_k} + \nabla h\p{z_k}}{x - z_k} \\
    &= \p{f + h}\p{x} - \frac{1}{\gamma_k} \inpr{x_k - x_{k+1}}{x - z_k}.
  \end{align*}
On the other hand, for any $k \geq 0$ it holds
\begin{align*}
    \inpr{x_{k+1} - x_k}{x - z_k} &= \frac{1}{2}\p{\norm{x_{k+1} - z_k}^2 + \norm{x_k - x}^2 - \norm{x_{k+1} - x}^2 - \norm{x_k - z_k}^2} \\
    &\leq \frac{1}{2} \p{\norm{x_k - x}^2 - \norm{x_{k+1} - x}^2} - \frac{1 - \frac{\gamma_k^2}{\beta^2}}{2} \norm{x_k - z_k}^2 \\
    &\leq \frac{1}{2} \p{\norm{x_k - x}^2 - \norm{x_{k+1} - x}^2},
  \end{align*}
thus
$$\p{f + h}\p{z_k} \leq \p{f + h}\p{x} + \frac{1}{2\gamma_k} \p{\norm{x_k - x}^2 - \norm{x_{k+1} - x}^2}.$$ 
Setting 
$$\Gamma_n \defeq \sum_{k = 0}^n \gamma_k \ \mbox{and} \ \zeta_n \defeq \frac{1}{\Gamma_n} \sum_{k = 0}^n \gamma_k z_k,$$ we obtain, by Jensen's inequality in discrete form, that
  \begin{align*}
    \p{f + h}\p{\zeta_n} &\leq \frac{1}{\Gamma_n} \sum_{n = 0}^n \gamma_k \p{f + h}\p{z_k} \\
    &\leq \p{f + h}\p{x} + \frac{1}{2\Gamma_n}\sum_{k = 0}^n \p{\norm{x_k - x}^2 - \norm{x_{k+1} - x}^2}\\
    & = \p{f + h}\p{x} + \frac{1}{2\Gamma_n} \p{\norm{x_0 - x}^2 - \norm{x_{n+1} - x}^2}\\
    & \leq \p{f + h}\p{x} + \frac{\norm{x_0 - x}^2}{2\Gamma_n}.
  \end{align*}
  
\end{remark}

\end{document}